\documentclass[11pt,a4paper]{amsart}
\usepackage[all]{xy}
\usepackage{amscd}
\usepackage{xcolor}
\usepackage{latexsym}
\usepackage{amsmath,amssymb,amsthm}

\newtheorem{theorem}{Theorem}[section]
\newtheorem{lemma}[theorem]{Lemma}
\newtheorem{corol}[theorem]{Corollary}
\newtheorem{prop}[theorem]{Proposition}

\theoremstyle{definition}
\newtheorem{defin}[theorem]{Definition}
\newtheorem{exam}[theorem]{Example}
\theoremstyle{remark}
\newtheorem*{erem}{Remark}

 \def\={\setminus}

\def\md#1{#1\mbox{-}\mathrm{mod}}
\def\pr#1{#1\mbox{-}\mathrm{pro}}

\def\bp{\bullet}    \def\iso{\simeq}

\def\add{\mathop\mathrm{add}\nolimits}

\def\Dim{\mathop\mathbf{dim}}
\def\iff{if and only if }

\def\setsuch#1#2{\left\{\,#1\,|\,#2\,\right\}}

\def\rad{\mathop\mathrm{rad}}

\def\Ext{\mathop\mathrm{Ext}\nolimits}

\def\Hom{\mathop\mathrm{Hom}}
\def\End{\mathop\mathrm{End}}

\def\im{\mathop\mathrm{Im}}
\def\8{\infty}      \def\+{\oplus}
\def\*{\otimes}

\def\al{\alpha} \def\be{\beta}  \def\ga{\gamma}
       \def\la{\lambda}
     
 \def\Si{\Sigma}

\def\bA{\mathbf A}  \def\bB{\mathbf B}
  
\def\fR{\mathbf r}

\def\bC{\mathbf C}

\def\mN{\mathbb N}  \def\mZ{\mathbb Z}
\def\mA{\mathbb A}  
 
\def\mC{\mathbb C} \def\mL{\mathbb L}

 \def\kD{\mathcal D}
\def\kK{\mathcal K} 
 \def\kQ{\mathcal Q}

 \def\kI{\mathcal I}

    \def\dP{\mathfrak P}
    
\def\gS{\mathfrak s}    \def\gT{\mathfrak t}

      \def\sJ{\mathsf J}

\def\sR{\mathsf R}

\def\Mk{\Bbbk}
\def\bA{\mathbf A}

\begin{document}
\title{Derived tame Nakayama algebras}

\author{Viktor Bekkert, Hern\'{a}n Giraldo \and Jos\'{e} A. V\'{e}lez-Marulanda}

\address{Departamento de Matem\'atica, ICEx, Universidade Federal de Minas
Gerais, Av.  Ant\^onio Carlos, 6627, CP 702, CEP 30123-970, Belo
Horizonte-MG, Brasil}
\email{bekkert@mat.ufmg.br}
\address{Instituto de Matem\'{a}ticas, Universidad de Antioquia, Medell\'{\i}n, Antioquia, Colombia}
\email{hernan.giraldo@udea.edu.co}
\address{Department of Mathematics, Valdosta State University, Valdosta, GA, USA}
\email{javelezmarulanda@valdosta.edu}


\begin{abstract}
{We determine the derived representation type of Nakayama algebras and
 prove that a derived tame Nakayama algebra without simple projective module
 is gentle or derived equivalent to some skewed-gentle algebra, and as a consequence,
 we determine its  singularity category.}
\end{abstract}

\maketitle

\section{Introduction}
Let $\Mk$ be an algebraically closed field of arbitrary characteristic, let $\bA$ be
a finite dimensional $\Mk$-algebra and let
$\kD^{b}(\md{\bA})$ be the bounded derived category of the category of finitely generated modules $\md{\bA}$. 
One of the  main problems in the representation theory of algebras is a classification of indecomposable finitely generated modules.
The dichotomy theorem of Drozd \cite{d0} divides all finite dimensional algebras according to their {\em representation type} into {\em tame} and {\em wild}.
In the case of tame algebras a classification of indecomposable modules is relatively easy, in the sense that 
for each dimension $d$ they admit a parametrization of $d$-dimensional indecomposable modules by a finite number of $1$-parameter families. The situation is much more complicated for wild algebras. This singles out the problem of establishing the representation type of a given algebra. 

During the last years there has been an active study of derived categories. In particular, a notion 
of {\em derived representation type} was introduced for finite dimensional algebras \cite{gk}. 
The tame-wild dichotomy for  derived categories over finite dimensional algebras was established in \cite{bd} (see also \cite{d2, ba}). 
The structure of the derived category is known for a few classes of finite dimensional algebras (see e.g. \cite{bm, bmm, bud2, h, hr}).

In this paper, we are interested in the case when $\bA$ is derived-tame.
The derived representation type is well-known for tree algebras \cite{b, g},
for blowing-up of tree algebras \cite{c}, 
for algebras  with radical square zero \cite{bd, bl},  and for nodal algebras \cite{bud2}.

Recall that $\bA$ is said to be a {\it Nakayama algebra}
if every left and 
right indecomposable projective $\bA$-module has a unique composition
series. It is well-known (see e.g. \cite[Thm. V.3.2]{as2} and \cite{ku}) that $\bA$ is a connected basic Nakayama algebra with $n$ non-isomorphic simple
modules if and only if $\bA$ is isomorphic to a bound quiver
algebra $\Mk \kQ/\kI$, where $\Mk \kQ$ is the path algebra of a quiver $\kQ$, which is of one of the following
two types:

 \[
 \mL_n: \xymatrix{ 1 \ar[r]^{a_1} & 2 \ar[r]^{a_2} & \cdots  \ar[r]^{a_{n-1}} & n},
 \]

 \[
  \mC_n: \xymatrix{ 0 \ar[r]^{a_0} & 1 \ar[r]^{a_1} & \cdots  \ar[r]^{a_{n-2}} &  n-1  \ar@/^1pc/[lll]^{a_{n-1}} },
 \]
 
\medskip\noindent
and $\kI$ is an admissible ideal of $\Mk \kQ$. In this case, there is a minimal
set of paths in $\kQ$ generating the ideal $\kI$, which will be denoted by $R_{\bA}$. 
A path $a_ia_{i+1}\cdots a_{i+m-1}\in R_{\bA} $ is called a \emph{minimal $m$-relation}.
We will asume also that
$\kQ_0=\{1,2,\dots, n-1, n\}$ in the case of $\kQ=\mL_n$, and $\kQ_0=\mZ/(n)$ in the case of $\kQ=\mC_n$.
Following \cite{cy} we call a Nakayama algebra $\Mk \kQ/\kI$ a \emph{line algebra} in case of $\kQ=\mL_n$,
and a \emph{cycle algebra} in case of $\kQ=\mC_n$.

We introduce a class $\kD$ of Nakayama algebras as follows. 
\begin{defin}\label{condNakayama}
A finite dimensional, basic
and connected Nakayama algebra $\bA=\Mk\kQ/\kI$ belongs to $\kD$ if the following conditions are
satisfied:

\begin{enumerate}
\item[(C1)] The ideal $\kI$ is generated by a set  of paths of length two or three.

\item[(C2)] If $a_ia_{i+1}a_{i+2}\in R_{\bA}$ then $a_{i+1}a_{i+2}a_{i+3}\in R_{\bA}$ or $a_{i-1}a_{i}a_{i+1}\in R_{\bA}$. 

\item[(C3)] If $a_ia_{i+1}a_{i+2}, a_{i+1}a_{i+2}a_{i+3}\in R_{\bA}$  then $a_{i-1}a_{i}a_{i+1}, a_{i+2}a_{i+3}a_{i+4}\notin R_{\bA}.$

\end{enumerate}
We call a minimal $3$-relation {\em isolated} if it does not satisfy the condition (C2).
\end{defin}

Recall that, if $\bA$ is an algebra of finite global dimension, then its Euler quadratic form is defined on the
Grothendieck group of $\bA$ by $\chi_{\bA}(\Dim M)=\sum_{i=0}^{\infty}(-1)^{i}\Dim_\Mk\Ext_{\bA}^{i}(M,M)$ for any $\bA$-module $M$. 

Our main results are the following theorems.

\begin{theorem} \label{t1} Let $\bA$ be a basic connected Nakayama algebra.
Then $\bA$ is derived tame if and only if one of the following conditions holds:

\begin{enumerate}
 \item[(i)] $\bA$ is line algebra and its Euler form is non-negative;
 \item[(ii)] $\bA$ belongs to the class $\kD$.
\end{enumerate}

\end{theorem}

\begin{theorem} \label{t2}  Let $\bA$ be a cycle  algebra.
Then $\bA$ is derived tame if and only if $\bA$ is gentle or derived equivalent
to some skewed-gentle algebra.
\end{theorem}

\begin{erem}
 \begin{enumerate}
  \item In the case of line algebras, Theorem~\ref{t1} follows from \cite{b, g}.
  \item Derived equivalence classification of Nakayama algebras is known for the case
  of line algebras \cite{b, g}, and for the case of gentle cycle algebras \cite{bgs}.
  \item The piecewise heredity of truncated line algebras have been studied in \cite{hs}.
 \end{enumerate}

\end{erem}

Following \cite{bu, ha, o}, the {\em singularity category} $\kD_{sg}(\bA)$ of $\bA$ is the Verdier quotient
of $\kD^{b}(\md{\bA})$ with respect to the full triangulated subcategory consisting of perfect complexes (see \cite{ver} or e.g. \cite{k2} for the construction of the quotient category).
 As a corollary of Theorem~\ref{t2} and \cite{ka, cl},
 we determine the singularity category of a derived tame cycle algebra (cf. \cite{cy}).
 
 \begin{corol}\label{c1}  Let $\bA$ be a derived tame cycle algebra. Then there is an equivalence of
 triangulated categories
 
 $$\kD_{sg}(\bA)\cong \frac{\kD^{b}(\md{\Mk})}{[| R_{\bA} |]}, $$
 
\noindent where $\frac{\kD^{b}(\md{\Mk})}{[| R_{\bA} |]}$ denotes the triangulated orbit category in the sense of \cite{ke}.
  
 \end{corol}

 An immediate consequence of Corollary \ref{c1} is the following result.
 
 \begin{corol}\label{c2}
  Suppose that $\bA$ and $\bB$ are derived equivalent derived tame cycle algebras.
  Then $| R_{\bA} |=| R_{\bB} |$.
  
 \end{corol}

The structure of this article is as follows. In Section~\ref{s1},
we review some preliminary results about derived categories and derived representation type classification of finite dimensional $\Mk$-algebras. Moreover, we also recall the definitions of gentle and skewed-gentle algebras as initially provided in \cite{as} and \cite{gp}, respectively.  In Section~\ref{s2}, we prove Theorems~\ref{t1} and  \ref{t2}, and Corollary~\ref{c1}.

\section{Preliminaries}\label{s1}

\subsection{Derived representation type}\label{s11}

Let $\bA$ be an associative finite dimensional  $\Mk$-algebra.
We denote by  $\md{\bA}$  the category of left finitely
generated $\bA$-modules, by $\kD(\bA)$  its derived category, and by $\kD^{b}(\md{\bA})$ the derived category of
bounded complexes whose terms are in $\md{\bA}$. It is well-known that $\kD^{b}(\md{\bA})$ can be identified with the
homotopy category $\kK^{-,b}(\pr{\bA})$ of bounded above complexes of finitely generated projective
$\bA$-modules with bounded homologies. Recall that every object in $\kK^{-,b}(\pr{\bA})$ is homotopy equivalent to a \emph{minimal}
one (see e.g. \cite{k} and \cite{gm}), i.e., to  a complex $C_{\bullet}=(C_n, d_n)$ such that $\im d_n\subseteq \rad C_{n-1}$ for all $n$. If $C_{\bullet}$
and $C^{\prime}_{\bullet}$ are two minimal complexes, then they are isomorphic in $\kD(\bA)$ if and only if they are isomorphic as complexes.
Moreover, any morphism $f: C_{\bullet}\to C^{\prime}_{\bullet}$ in $\kD(\bA)$ can be presented by a morphism of complexes, and
$f$ is an isomorphism if and only if the latter one is. For convenience, we write composition of morphisms from left to right.

Let $A_1, A_2, \ldots, A_t$ be all pairwise non-isomorphic indecomposable projective $\bA$-modules (all of them are direct summands of $\bA$).
If $P$ is a finitely generated projective $\bA$-module, then $P$ decomposes uniquely as 
$$P=\bigoplus_{i=1}^{t}p_iA_i,$$
where for all $1\leq i\leq t$,  $p_i$ is a non-negative integer.  
Denote by $\fR(P)$ the
vector $(p_1,p_2,\ldots,p_t)$. Let $P_\bullet=(P_n,d_n)$ be a bounded complex whose terms are finitely generated projective $\bA$-modules. The sequence $$(\dots,\fR(P_n),\fR(P_{n-1}),\dots)$$ (it
 has only finitely many nonzero entries)
is called the \emph{vector rank} $\fR_\bp(P_\bp)$
 of $P_\bp$.

The following definition provides a version of derived tameness and wildness from \cite{bd} for finite dimensional algebras.
  \begin{defin}\label{tw}
  \begin{enumerate}
 \item[(i)]
  We call a \emph{rational family} of bounded minimal complexes over $\bA$ a bounded complex $(P_\bp,d_\bp)$
 of finitely generated projective $\bA\*\sR$-modules, where $\sR$ is a \emph{rational algebra},
 i.e. $\sR=\Mk[t,f(t)^{-1}]$ for a nonzero polynomial $f(t)$, and $\im d_n\subseteq\sJ P_{n-1}$, where $\sJ=\rad\bA$.
 For a  rational family $(P_\bp,d_\bp)$ we define the complex $P_\bp(m,\la)=(P_\bp\*_\sR\sR/(t-\la)^m,d_\bp\*1)$ of projective $\bA$-modules, where $m\in\mN,\,\la\in\Mk,\,f(\la)\ne0$. Set $\fR_\bp(P_\bp)=
 \fR_\bp(P_\bp(1,\la))$ ($\fR_\bp$ does not depend on $\la$).
 \item[(ii)]
  We call $\bA$ \emph{derived tame} if there is a set $\dP$ of rational families of bounded complexes over
 $\bA$ such that:
  \begin{enumerate}
 \item[(ii.a)]  for each vector rank $\fR_\bp$ the set $\dP(\fR_\bp)=\setsuch{P_\bp\in\dP}{\fR_\bp(P_\bp)=\fR_\bp}$ is finite;
 \item[(ii.b)]
  for each vector rank $\fR_\bp$ all indecomposable complexes $(P_\bp,d_\bp)$ of projective $\bA$-modules
 of this vector rank, except finitely many isomorphism classes, are isomorphic to $P_\bp(m,\la)$ for some $P_\bp\in\dP$
 and some $in \mathbb{N}$ and $\la\in \Mk$.
\end{enumerate}
 The set $\dP$ is called a \emph{parameterizing set} of $\bA$-complexes.
 \item[(iii)]
  We call $\bA$ \emph{derived wild} if there is a bounded complex $(P_\bp,d_\bp)$ of projective modules over
$\bA\*\Si$, where $\Si$ is the free $\Mk$-algebra in 2 variables,
such that $\im d_n\subseteq\sJ P_{n-1}$ and, for any finite dimensional $\Si$-modules $L,L'$:
  \begin{enumerate}
\item[(iii.a)]   $P_\bp\*_\Si L\iso P_\bp\*_\Si L'$ \iff $L\iso L'$; and 
 \item[(iii.b)]
  $P_\bp\*_\Si L$ is indecomposable \iff so is $L$.
\end{enumerate}
\end{enumerate}
 \end{defin}

Note that, according to Definition \ref{tw}, every \emph{derived discrete} (in particular, \emph{derived finite}) algebra
 \cite{v} is derived tame (with the empty set  $\dP$). On the other hand, it is proved in \cite{bd} that every finite dimensional algebra over an
algebraically closed field is either derived tame or derived wild.

\subsection{Quivers with relations}\label{s12}

A \emph{quiver} $\kQ$ is a tuple $(\kQ_0,\kQ_1,\gS,
\gT)$ consisting of a set $\kQ_0$ of \emph{vertices}, a
set $\kQ_1$ of \emph{arrows}, and maps $\gS,\gT: \kQ_1
\rightarrow \kQ_0$ which specify the \emph{starting} and \emph{ending} vertices for each arrow $a\in \kQ_1$.
Given two vertices $i$ and $j$ we define $\kQ_1[i,j]$ as the set of all arrows
from $i$ to $j$.
A \emph{path} $p$ in $\kQ$ of length $\ell(p) = n \geq 1$ is a
sequence of arrows $a_1,\dots,a_n$ such that
$\mathfrak{s}(a_{i+1}) = \mathfrak{t}(a_i)$ for $1 \leq i < n$.
In this situation, we set $\mathfrak{s}(p) = \mathfrak{s}(a_1)$ and $\mathfrak{t}(p) =
\mathfrak{t}(a_n)$. Note that we write paths from left to right  for convenience. On the other hand, the concatenation $pp^{\prime}$ of two paths $p$,
$p^{\prime}$ in $\kQ$ is defined in the natural way whenever $\gS(p')=\gT(p)$. Every vertex $i \in \kQ_0$ determines a path $e_i$ (of length $0$) with $\gS(e_i) = i$ and
$\gT(e_{i}) = i$.
A quiver $\kQ$ determines  the path algebra $\Mk\kQ$, which has a
$\Mk$-basis consisting of all the paths in $\kQ$, and the multiplication is given by the path-concatenation provided that exists, or zero otherwise. The algebra $\Mk\kQ$ is finite-dimensional precisely when $\kQ$ does not contain an oriented cycle.
An ideal $\kI\subseteq \Mk\kQ$ is called \emph{admissible} if $\kI\subseteq \rad^{2}(\Mk\kQ)$ where $\rad(\Mk\kQ)$ is the radical of the algebra $\Mk\kQ$.
It is well-known that if $\Mk$ is algebraically closed, any finite-dimensional $\Mk$-algebra is Morita
equivalent to a quotient $\Mk\kQ/\kI$ where $\kI$ is an admissible ideal.
By a slight abuse of notation, we identify paths in the quiver $\kQ$ with their 
cosets in $\Mk\kQ/\kI$.

We denote by $R_{\bA}^{m}$ (resp., $R_{\bA}^{\geq m}$,  resp., $R_{\bA}^{\leq m}$ )   the set of
minimal relations in $R_{\bA}$ of length $m$ (resp., greater or equal that $m$,  resp., 
less or equal that $m$).

For a vertex $i\in \kQ_0$ such that $a_{i-1}a_i\neq 0$, 
let $e=\sum_{j\in\kQ_{0}\setminus \{i\}}e_j$ and 
we denote by $\bA(i)$ the full subalgebra $e\bA e$ of the algebra $\bA$.
Then we can assume that
$\bA(i)=\Mk\kQ_{\bA(i)}/\kI_{\bA(i)}$, where 
$(\kQ_{\bA(i)})_0=\kQ_0\setminus \{i\}$, $(\kQ_{\bA(i)})_1=(\kQ_1\setminus \{a_{i-1},a_{i}\})\bigcup \{g=a_{i-1}a_{i}\}$
and $\kI_{\bA(i)}=\kI\bigcap \bA(i)$. 

\subsection{Gentle algebras}\label{s13}

Let $\kQ$ be a quiver and let $\kI$ be an admissible ideal in the path algebra  $\Mk\kQ$.
\begin{defin}
 \label{specialbi}
The pair $(\kQ,\kI)$ is said to be \emph{special biserial}  \cite{but,sw} if
the following conditions hold.
\begin{itemize}
\item[(G1)] At every vertex of $\kQ$ at most two arrows end and at most two
arrows start.
\item[(G2)] For each arrow $b$ there is at most one arrow $a$ with
$\gT(a)=\gS(b)$ and $ab\not\in \kI$ and at most one arrow $c$ with
$\gT(b)=\gS(c)$ and $bc\not\in \kI$.
\end{itemize}
\end{defin}

\begin{defin}\label{defgentle}
The pair $(\kQ,\kI)$ is said to be \emph{gentle} \cite{as} if
it is special biserial, and moreover the following conditions hold.

\begin{itemize}
\item[(G3)] $\kI$ is generated by zero relations of length 2.
\item[(G4)] For each arrow $b$ there is at most one arrow $a$ with
$\gT(a)=\gS(b)$ and $ab\in \kI$ and at most one arrow $c$ with
$\gT(b)=\gS(c)$ and $bc\in \kI$.
\end{itemize}

A $\Mk$-algebra $\bA$ is called \emph{gentle} \cite{as},
if it is Morita-equivalent to a factor algebra $\Mk\kQ/\kI$,
where the par $(\kQ,\kI)$ is gentle.
\end{defin}

The next theorem follows from \cite{ps} (see also \cite{bm}).

\begin{theorem}\label{thm-gentle-tame}
Any gentle algebra is derived tame.

\end{theorem}

\subsection{Skewed-gentle algebras}\label{s14}

Let $\kQ$ be a quiver with a fixed distinguished set of
vertices, which we denote by $Sp$, and $R$ a set of relations for $\kQ$.
We call the elements of $Sp$ \emph{special vertices}, and the
remaining vertices are called \emph{ordinary}.

For a triple $(\kQ, Sp, R)$, we consider the pair $(\kQ^{sp}, R^{sp})$,
where $\kQ^{sp}_{0}:=\kQ_0$,
$\kQ^{sp}_1:=\kQ_1\cup \{a_i\,|\,i\in Sp\}$, $\gS(a_i):=\gT(a_i):=i$
and $R^{sp}:=R\cup \{a^{2}_i\,|\,i\in Sp\}$.

\begin{defin}\label{def-skewed-gentle-triple}
A triple $(\kQ, Sp, R)$ as above is called \emph{skewed-gentle}
if the corresponding pair $(\kQ^{sp}, \left< R^{sp}\right>)$ is gentle.
\end{defin}

Let $(\kQ, Sp, R)$ be a skewed-gentle triple. 
We associate to each vertex $ i \in \kQ_0$  a set, which we will denote by $\kQ_0(i)$ in the following way.
If $i\not\in Sp$, then $\kQ_0(i) = \{i\}$, and if $i\in Sp$, then $\kQ_0(i) = \{i^{-},i^{+}\}$.
The quiver with relations $(\kQ^{sg},R^{sg})$ is defined in the following way:
$$\kQ^{sg}_0:=\cup_{i\in \kQ_0}\kQ_0(i),$$
$$\kQ^{sg}_1[\al, \be]:=\{(\al, a, \be)\,|\,a\in \kQ_1, \al\in \kQ_0(\gS(a)),
\be\in \kQ_0(\gT(a))\},$$
$$R^{sg}:=\left\{\sum_{\be\in \kQ_0(\gS(b))}\la_\be(\al, a, \be)(\be, b, \ga)\,|\,
ab\in R, \al\in \kQ_0(\gS(a)), \ga\in \kQ_0(\gT(b))\right\},$$
\noindent where $\la_\be=-1$ if $\be=i^{-}$ for some $i\in \kQ_0$,
and $\la_\be=1$ otherwise.
\noindent
Note that the relations in $R^{sg}$ are zero-relations or
commutative relations.
We denote by $a^{-}$ (resp., $a^{+}$) the arrows of the form $(i^{-},a,j)$ or $(i,a,j^{-})$
(resp., $(i^{+},a,j)$ or $(i,a,j^{+})$).

\begin{defin}\label{def-skewed-gentle}
A $\Mk$-algebra $\bA$ is called \emph{skewed-gentle} \cite{gp},
if it is Morita-equivalent to a factor algebra $\Mk\kQ^{sg}/\left< R^{sg}\right>$,
where the triple $(\kQ, Sp, R)$ is skewed-gentle.
\end{defin}

The next theorem follows from \cite{gp} (see also \cite{bmm}).

\begin{theorem}\label{thm-skewed-gentle-tame}
Any skewed-gentle algebra is derived tame.

\end{theorem}

\subsection{Derived equivalence and derived tameness}\label{s15}

We recall that if for finite dimensional $\Mk$-algebras $\bA$ and $\bB$, the derived
categories $\kD^{b}(\md{\bA})$ and $\kD^{b}(\md{\bB})$ are equivalent as triangulated categories,
then $\bA$ and $\bB$  are  said to be \emph{derived equivalent}. By a fundamental result due to Rickard \cite{r}, this
happens exactly when there exists a complex  $T_\bp$ in $\kK^{b}(\pr{\bA})$ (called a \emph{tilting complex})
with the following properties:
\begin{itemize}
 \item[(i)] $\Hom_{\kD^{b}(\md{\bA})}(T_\bp, T_\bp[i])=0$ for $i\neq 0$ (where $[-]$ denote the shift functor);
 \item[(ii)] $\add(T_\bp)$, the full subcategory of $\kK^{b}(\pr{\bA})$ consisting of direct summands of direct sums of
 copies of $T_\bp$, generates $\kK^{b}(\pr{\bA})$ as a triangulated category;
 \item[(iii)] $\bB\cong \End_{\kD^{b}(\md{\bA})}(T_\bp)$.
\end{itemize}

We recall the following result from \cite{gk}.

\begin{theorem}\label{dereq}
Derived tameness is preserved under derived equivalence.

\end{theorem}

\section{Classification}\label{s2}

\subsection{Derived wildness}\label{s21}

The following technical lemmata are needed for the proofs of the main theorems.
 
\begin{lemma}\label{subalg}\cite[Lemma 3.1]{bdf}
Let  $\bB$ be a full subalgebra of $\bA $
(i.e., a subalgebra of the form $e\bA e$ for some idempotent $e$).
If $\bB$ is derived wild then $\bA$ is derived wild.
\end{lemma}

 We now define a special class of cycle truncated $\Mk$-algebras as follows.
 Let $n>0, r\geq 2$ and set $\bC(n,r)=\Mk \kQ/ \kI$, where $\kQ=\mC_n$ and
 $\kI$ is generated by the set $\{a_ia_{i+1}\cdots a_{i+r-1}\, |\, i\in \kQ_0\}$.
 Note that a similar class of line algebras has been investigated in \cite{hs}.

 \begin{lemma}\label{selfinjective}
 Let $\bA=\bC(n,r)$  be a cycle truncated algebra.
Then $\bA$ is  derived tame if and only if $r=2$.
\end{lemma}

\begin{proof} 
If $r=2$, $\bA$ is a gentle algebra, then is derived tame by Theorem  \ref{thm-gentle-tame}. Assume next that $r\geq 3$.
It was proved in \cite[Cor. 2.5]{ba} that if $\bA$ is self-injective then
$\bA$ is either derived discrete (see \cite{v}) or derived wild. 
Since the algebras $\bA=\bC(n,r)$ are all self-injective and not derived discrete by \cite{v}, the statement follows. 

\end{proof}

Note that Lemma \ref{selfinjective} was also proved in \cite[Prop. 3.1]{zh} by using cleaving functors.

\begin{lemma}\label{derwild2}
Let $\bA=\Mk \kQ/\kI$ be a  cycle algebra
satisfying the condition (C1) but not satisfying the condition (C2)
in Definition~\ref{condNakayama}. Then $\bA$ is  derived wild.
\end{lemma}
\begin{proof} Without loss of generality, we can assume that $a_0a_1a_2$ is an isolated $3$-relation.
Note that this means that $n>1$. 

We consider first a particular case, and then the general case.

\begin{itemize}
 \item[(a)] Assume first that $R_{\bA}=\{a_0a_1a_2, a_{i-1}a_i \,\,|\,\,
 i\in\kQ_{0}\setminus \{1,2\}\}$. Define a complex $T_\bp=\oplus_{i\in \kQ_{0}} T_i$ of $\bA$-modules as follows: 
 $$T_i:0\to A_i\to 0 \quad (\textnormal{in degree}\,\, 0) \quad\textnormal{for} \,\, i\in \kQ_0, i\neq 1,$$
  $$T_1:\xymatrix{ 0 \ar[r] & A_1 \ar[r]^{a_1} &  A_2 \ar[r] & 0}  \quad (\textnormal{in degrees}\,\, 1\,\,\textnormal{and}\,\, 0).$$ 
 It is easy to check that the complex $T_\bp$ is tilting  and  the
endomorphism algebra $\End_{\kD^{b}(\md{\bA})}(T_\bp)$ is isomorphic to  the algebra $\bB=\Mk \kQ_{\bB}/\kI_{\bB}$ which can be obtained from $\bA$ as follows:
the quiver $\kQ_{\bB}$ is obtained from $\kQ$ by replacing the subquiver 

\[
  \xymatrix{ 0 \ar[r]^{a_{0}} & 1 \ar[r]^{a_1} &  2 \ar[r]^{a_{2}} &  3 }
 \]
 
\noindent for the subquiver of the form

 \[ \xymatrix{ & 1 & \\
  0 \ar[r]^{b} & 2 \ar[r]^{a_2}\ar[u]^{c} & 3  
}
\]

\noindent and $\kI_{\bB}$ is the ideal of $\Mk \kQ_{\bB}$ generated by the set $R_{\bB}$ obtained 
from $R_{\bA}$  by replacing  the relation $a_{0}a_{1}a_{2}$
for the relations $a_{n-1}b$, $bc$ and $ba_2$, and keeping the remaining elements of $R_{\bA}$. Since $\bB$ is an algebra with radical square zero and $\kQ_{\bB}$ is neither of Dynkin nor of Euclidean type, it follows from \cite[Thm. 3.1]{bd} that $\bB$ is derived wild and hence $\bA$ is derived wild by Theorem \ref{dereq}.

\item[(b)] Because of (a) it remains to prove the result in the case that there exists $i\in \kQ_{0}\setminus \{1, 2\}$ such that $a_{i-1}a_i\notin R_{\bA}$ and $n>2$. 
Let  $\bB=\bA(i)$. Then $\bB$ is cyclic, for $a_{i-1}a_i\notin R_{\bA}$. 
Let 
$$
b =\begin{cases}
a_0,& \text{if } i\neq 0,\\
a_{n-1}a_0,& \text{if } i=0,\\
\end{cases} \;\;\;\;\;
c =\begin{cases}
a_2,& \text{if } i\neq 3,\\
a_{2}a_3,& \text{if } i=3,\\
\end{cases}$$

$$
d =\begin{cases}
a_{n-1},& \text{if } i\notin \{n-1, 0\},\\
a_{n-2},& \text{if } i=0,\\
a_{n-2}a_{n-1},& \text{if } i=n-1,\\
\end{cases} \;\;\;\;\;
f =\begin{cases}
a_{3},& \text{if } i\notin \{3, 4\},\\
a_{4},& \text{if } i=3,\\
a_{3}a_{4},& \text{if } i=4.\\
\end{cases}$$

It is easy to see that $b,c,d,f\in (\kQ_{\bB})_1$ in any of these cases.
We  prove that $ba_1c$ is an isolated $3$-relation in the algebra $\bB$. For this, we need to show that $ba_1c\in R_{\bB}$ and 
$dba_1, a_1cf\notin R_{\bB}$. 
\noindent
Since $a_0a_1a_2\in R_{\bA}$, then $ba_1c\in \kI_{\bB}$.
If $i\neq 0$, then $ba_1=a_0a_1$, and thus $ba_1\notin \kI_{\bB}$.
If $i=0$, then $ba_1=a_{n-1}a_0a_1$ and $a_{n-1}a_0\notin \kI_{\bA}$ and since $a_0a_1a_2$ is an isolated $3$-relation, $ba_1\notin \kI_{\bB}$.
In a similar way we can show that $a_1c\notin \kI_{\bB}$. 
Therefore, $ba_1c\in R_{\bB}$.
Suppose that $dba_1\in R_{\bB}$. If $i=0$, then $d=a_{n-2}$ and
$b=a_{n-1}a_{0}$. Since $dba_1\in R_{\bB}$, we have $a_{n-2}a_{n-1}a_0a_1\in \kI_{\bA}$ and $db=a_{n-2}a_{n-1}a_{0}, ba_1=a_{n-1}a_0a_1\notin \kI_{\bA}$, which implies $a_{n-2}a_{n-1}a_0a_1\in R_{\bA}$ and which contradicts the condition (C1). The case $i=n-1$ is analogous. If $i\notin \{n-1, 0\}$, then $d=a_{n-1}$ and $b=a_0$.
Since $dba_1\in R_{\bB}$, we have $a_{n-1}a_0a_1\in \kI_{\bA}$ and $a_{n-1}a_0=db, a_0a_1=ba_1\notin \kI_{\bA}$, and thus $a_{n-1}a_0a_1\in R_{\bA}$, which is a contradiction for $a_0a_1a_2$ is an isolated $3$-relation. Therefore $dba_1\notin R_{\bB}$.
In the similar way we can show that $a_1cf\notin R_{\bB}$.
Since $ba_1c\in R_{\bB}$ and 
$dba_1, a_1cf\notin R_{\bB}$, we have that $ba_1c$ is isolated $3$-relation in the algebra $\bB$.
The result now follows by induction on $n$ because of Lemma \ref{subalg} and if $n=2$, then we are clearly in  the case (a).
\end{itemize}
\end{proof}

 \begin{lemma}\label{derwild3}
Let $\bA=\Mk \kQ/\kI$ be a  cycle algebra
satisfying the the condition (C1) and not satisfying the condition (C3)
in Definition~\ref{condNakayama}.
Then $\bA$ is  derived wild.
\end{lemma}

\begin{proof} 
If $a_{i-1}a_{i}a_{i+1}\in R_{\bA}$ for all $i\in \kQ_0$, then
$\bA=\bC(n,3)$ is a cycle truncated algebra, where $n$ is the number of vertices of $\kQ$.
In particular, this happens provided that $n<4$. Hence $\bA$ is 
derived wild in these cases by Lemma~\ref{selfinjective}. Therefore we can assume
without loss of generality that $n\geq 4$, $a_1a_2a_3$, $a_2a_3a_4$, $a_3a_4a_5\in R_{\bA}$ and
that $a_0a_1a_2\notin R_{\bA}$. Let $\bB=\bA(5)$ and 
$$
b =\begin{cases}
a_0a_1,& \text{if } n=4,\\
a_1,& \text{if } n>4,\\
\end{cases}$$

$$
c =\begin{cases}
a_0a_1,& \text{if } n=4,\\
a_4a_0,& \text{if } n=5,\\
a_4a_5,& \text{if } n>5,\\
\end{cases} \;\;\;\;\;
d =\begin{cases}
a_3,& \text{if } n=4,\\
a_{4}a_0,& \text{if } n=5,\\
a_0,& \text{if } n>5.\\
\end{cases}$$

\noindent It is easy to see that $b,c,d,f\in (\kQ_{\bB})_1$ in any of these cases. 
We show next that $ba_2a_3$ is an isolated $3$-relation in the algebra $\bB$, i.e., $ba_2a_3\in R_{\bB}$ and 
$dba_2, a_2a_3c\notin R_{\bB}$. To do this, we consider the cases $n=4,5$ and $n>5$.

\emph{Case $n=4$:} Since $a_1a_2a_3\in R_{\bA}$, it follows that $ba_2a_3\in \kI_{\bB}$ and $a_1a_2, a_2a_3\notin \kI_{\bA}$. On the other hand, since $a_0a_1a_2\notin R_{\bA}$ and $a_0a_1, a_1a_2\not\in I_{\bA}$, it follows that $a_0a_1a_2\notin \kI_{\bA}$ and thus $ba_2\notin \kI_{\bB}$. Since $ba_2a_3\in \kI_{\bB}$ and $ba_2, a_2a_3\notin \kI_{\bB}$, we obtain that $ba_2a_3\in R_{\bB}$, and since $a_3g \in R_{\bB}$, it follows that $dba_2, a_2a_3c\notin R_{\bB}$.

\emph{Case $n=5$:} Since $a_1a_2a_3\in R_{\bA}$ and $a_1,a_2,a_3\in (\kQ_{\bB})_1$, we have that $a_1a_2a_3\in R_{\bB}$. Suppose that $dba_2\in R_{\bB}$. Then $a_4a_0a_1a_2\in I_{\bA}$ and
$a_4a_0a_1\notin I_{\bA}$. If $a_0a_1a_2\notin \kI_{\bA}$, then $a_4a_0a_1a_2\in R_{\bA}$,  which contradicts the condition (C1). Therefore $a_0a_1\in \kI_{\bA}$ for $a_0a_1a_2\notin R_{\bA}$, which is a 
contradiction with $db\notin R_{\bB}$. Hence $dba_2\notin R_{\bB}$. Since $a_3a_4a_0\in R_{\bA}$, it follows that $a_3g\in R_{\bB}$ and thus $a_2a_3c\notin R_{\bB}$.

\emph{Case $n>5$:} Since $a_1a_2a_3\in R_{\bA}$ and $a_1,a_2,a_3\in (\kQ_{\bB})_1$, we obtain that $a_1a_2a_3\in R_{\bB}$, and since $a_0a_1a_2\notin R_{\bA}$ and $a_0,a_1,a_2\in (\kQ_{\bB})_1$, it follows that $dba_2\notin R_{\bB}$. On the other hand, since $a_3g\in R_{\bB}$, we have that $a_2a_3c=a_2a_3g\notin R_{\bB}$. 

Since in all these cases $ba_2a_3\in R_{\bB}$ and 
$dba_2, a_2a_3c\notin R_{\bB}$,  we obtain that $ba_2a_3$ is an 
isolated $3$-relation in the algebra $\bB$, and thus
$\bB$ is derived wild by Lemma~\ref{derwild2}. 
Therefore $\bA$ is derived wild by Lemma~\ref{subalg}.
\end{proof}

\begin{lemma}\label{derwild1}
Let $\bA=\Mk \kQ/\kI$ be a derived tame cycle algebra. 
Then the ideal $\kI$ can be generated by relations of length two
or three.
\end{lemma}

\begin{proof} Suppose that there exists a cycle derived tame algebra
$\bA=\Mk \kQ/\kI$ with some minimal relation 
$\rho=a_0a_1\cdots a_{m-1} \in R_{\bA}$ of length
$l(\rho)=m\geq 4$. We can assume that such $\bA$ has a minimal number $n$ of vertices.

\emph{Case $n\leq 2$:} Since there exists $\rho \in R_{\bA}$ with $l(\rho)\geq 3$, it follows 
by \cite{bdf} that $\bA$ is derived wild, which contradicts the derived tameness of $\bA$. 

\emph{Case $n=3$:} Let $\bB=\bA(1)$. Since $l(\rho)=m\geq 4$, we have a minimal
relation  of length greater that two in $R_{\bB}$. Hence $\bB$ is derived wild by
\cite{bdf} and therefore $\bA$ is derived wild by Lemma~\ref{subalg}, which is again a contradiction.

\emph{Case $m> 4$:} Because of the previous cases we can assume that $n>3$.
Let $\bB=\bA(1)$. Since $l(\rho)=m> 4$, we have a minimal
relation  of length greater that three in $R_{\bB}$. 
By \cite{bdf} if follows that $\bB$ is derived tame, which is a contradiction with the minimality of $n$.

From now we assume that $\rho=a_0a_1a_2a_3$, i.e., $m=4$ and that
 $R_{\bA}=R_{\bA}^{\leq 4}$.

\emph{Case $n\geq 4$:} Suppose that $a_{i-1}a_i\notin \kI$ for some
$i\in \kQ_0\setminus \{0, 1, \cdots ,4\}$. Let $\bB=\bA(i)$.
Since $\rho\in R_\bB$ and since $\bB$ is derived tame by Lemma~\ref{subalg},
we obtain a contradiction with the minimality of $n$. Hence we can assume
that $a_{i-1}a_i\in \kI$ for all $i\in \kQ_0\setminus \{0, 1, \cdots ,4\}$. Let $a_{n-1}a_0\notin \kI$ and $\bB=\bA(n-1)$. It follows from the minimality of $n$
that $ga_1a_2a_3\notin R_{\bB}$. Hence either $a_{n-1}a_0a_1\in R_{\bA}$ or 
$a_{n-1}a_0a_1a_2\in R_{\bA}$. Similarly, if  $a_3a_4\notin \kI$ then either 
$a_2a_3a_4\in R_{\bA}$ or $a_1a_2a_3a_4\in R_{\bA}$. Note that $a_{n-1}a_0=a_3a_4=a_3a_0$ for when
$n=4$. We next consider all the possibilities. 
\begin{itemize}
\item [(a)] If $a_{n-1}a_0$, $a_3a_4\in \kI$, then
$R_{\bA}^{\geq 3}=\{a_0a_1a_2a_3\}$. If $\bB=\bA(2)$, then $R_{\bB}^{\geq 3}=\{a_0ga_3 \}$
and hence $a_0ga_3$ is an isolated $3$-relation in $\bB$. 
\item[(b)] If $a_{n-1}a_0\in \kI$ and $ a_3a_4\notin \kI$, then either 
$R_{\bA}^{\geq 3}=\{a_0a_1a_2a_3, a_1a_2a_3a_4\}$ or
$R_{\bA}^{\geq 3}=\{a_0a_1a_2a_3, a_2a_3a_4\}$. Let $\bB=\bA(1)$ (resp., $\bB=\bA(3)$)  in the first case (resp., second case). Then $R_{\bB}^{\geq 3}=\{ga_2a_3 \}$ (resp., $R_{\bB}^{\geq 3}=\{a_0a_1g\}$)
and hence $ga_2a_3$  (resp., $a_0a_1g$) is an isolated $3$-relation in $\bB$.
\item[(c)] If $a_{n-1}a_0\notin \kI$ and $ a_3a_4\in \kI$, then we can argue as in the situation (b).

\item[(d)] Assume that $a_{n-1}a_0$, $a_3a_4\notin \kI$. We have the following
cases: 
\begin{itemize}
\item[(d.i)] $R_{\bA}^{\geq 3}\supseteq\{a_{n-1}a_0a_1a_2, a_0a_1a_2a_3,a_1a_2a_3a_4\};$
\item[(d.ii)] $R_{\bA}^{\geq 3}\supseteq\{a_{n-1}a_0a_1, a_0a_1a_2a_3, a_1a_2a_3a_4\};$
\item[(d.iii)] $R_{\bA}^{\geq 3}\supseteq\{a_{n-1}a_0a_1a_2, a_0a_1a_2a_3, a_2a_3a_4\}$;
\item[(d.iv)] $R_{\bA}^{\geq 3}\supseteq\{a_{n-1}a_0a_1, a_0a_1a_2a_3, a_2a_3a_4\}$.
\end{itemize}
In all the cases (d.i)-(d.iv), let $\bB=\bA(2)$. It follows that $R_{\bB}^{\geq 3}\supseteq
\{a_{n-1}a_0g, a_0ga_3, ga_3a_4 \}$ and hence we have three consecutive minimal
$3$-relations in $\bB$. Note that in all of these cases $R_{\bB}\subseteq R_{\bB}^{\leq 3}$,
i.e., $\bB$ satisfies the condition (C1) in Definition~\ref{condNakayama}.
\end{itemize}

In the situations (a)-(c), $\bB$ is derived wild by Lemma~\ref{derwild2}, whereas in
the situation (d), $\bB$ is derived wild by Lemma~\ref{derwild3}.
Hence $\bA$ is derived wild by Lemma~\ref{subalg}, 
which is again a contradiction. 
\end{proof}

\subsection{The class $\kD$}\label{s22}

Let $\bA=\Mk \kQ_{\bA}/\kI_{\bA}$ be an algebra
which belongs to class $\kD$. 
We set 
$\Omega=\{i\in (\kQ_{\bA})_0\, |\, a_{i-2}a_{i-1}a_{i}, 
a_{i-1}a_{i}a_{i+1} \in R_{\bA}\}$ and
$Sp=\{i\in (\kQ_{\bA})_0\, |\, i-1\in \Omega\}$.
Let $\bA^{\omega}$ be the full subalgebra $e\bA e$ of $\bA$, where
$e=\sum_{i\in (\kQ_{\bA})_0\setminus \Omega}e_i$.
Then we can assume that $\bA^{\omega}=\Mk \kQ_{\bA^{\omega}}/\kI_{\bA^{\omega}}$, 
where the quiver
$\kQ_{\bA^{\omega}}$ is obtained from $\kQ_{\bA}$ by replacing for each $i\in \Omega$
the subquiver $\xymatrix{ i-1 \ar[r]^{a_{i-1}} & i \ar[r]^{a_i} &  i+1} $
 by the quiver of the form $\xymatrix{ i-1  \ar[r]^{b_{i-1}} &  i+1} $,
  and $\kI_{\bA^{\omega}}$ is the ideal of $\Mk \kQ_{\bA^{\omega}}$ 
generated by the set $R_{\bA^{\omega}}$ obtained 
from $R_{\bA}$  by replacing for each $i\in \Omega$  the pair of 
relations $a_{i-2}a_{i-1}a_{i}, a_{i-1}a_{i}a_{i+1}$
for $a_{i-2}b_{i-1}, b_{i-1}a_{i+1}$.
It is easy to check that the algebra $\bA^{\omega}$ is gentle
(see Subsection~\ref{s13}) and thus $(\kQ_{\bA^{\omega}}, Sp, R_{\bA^{\omega}})$  is 
 a skewed-gentle triple 
(see Subsection~\ref{s14}). Then we denote by $\bA^{\Omega}$ the corresponding 
skewed-gentle algebra 
$\Mk (\kQ_{\bA^{\omega}})^{sg}/\left<(R_{\bA^{\omega}})^{sg}\right>$.

\begin{exam}\label{exam1}  Let $\bA=\Mk \kQ_{\bA}/\kI_{\bB}$ be 
the algebra such that 
\[
  \kQ_{\bA}=\mC_3: \xymatrix{ 0 \ar[rr]^{a_0} & &1 \ar[rr]^{a_1} & & 2    \ar@/^1pc/[llll]^{a_2} },
 \]
 
 \medskip\noindent and $\kI_{\bA}$ is the ideal of $\Mk \kQ_{\bA}$ generated 
by the set $R_{\bA}=\{a_0a_1a_2,$ $ a_1a_2a_0\}$. Then $\Omega=\{2\}$ and $\bA^{\omega}=\Mk \kQ_{\bA^{\omega}}/\kI_{\bA^{\omega}}$
 is the algebra such that

 \[
  \kQ_{\bA^{\omega}}: \xymatrix{
0 \ar@/^1pc/[rr]^{a_0} && 1 \ar@/^1pc/[ll]^{b_1}   },
 \]
 
 \medskip\noindent and $\kI_{\bA^{\omega}}$ is the ideal of $\Mk \kQ_{\bA^{\omega}}$ 
generated by the set $R_{\bA^{\omega}}=\{a_0b_1, b_1a_0\}$. Then  $(\kQ_{\bA^{\omega}}, Sp, R_{\bA^{\omega}})$, is  a skewed-gentle triple with $Sp=\{0\}$, and thus
 $\bA^{\Omega}=\Mk \kQ_{\bA^{\Omega}}/\kI_{\bA^{\Omega}}$ is the algebra such that

  \[
 \kQ_{\bA^{\Omega}}:
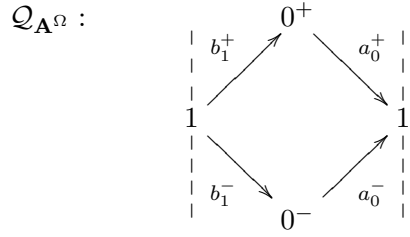
\xymatrix{
  &  &0^{+} \ar[rd]^{a_0^{+}}   &\\
 & 1 \ar@{--}[d] \ar@{--}[u]\ar[rd]_{b_{1}^{-}} \ar[ur]^{b_1^{+}} & & 1\ar@{--}[d] \ar@{--}[u] \\
   & & 0^{-}\ar[ur]_{a_0^{-}}    & }
\]

 \medskip\noindent where the dotted lines are identified, and $\kI_{\bA^{\Omega}}$ is the ideal of $\Mk \kQ_{\bA^{\Omega}}$ generated by the set  
 $R_{\bA^{\Omega}}=\{b_1^{+}a_0^{+}-b_1^{-}a_0^{-}, a_0^{\pm}b_1^{\pm}\}$. 
 
\end{exam}

\begin{exam}\label{exam2}  Let $\bA=\Mk \kQ_{\bA}/\kI_{\bB}$ be the algebra such that 
\[
  \kQ_{\bA}=\mC_4: \xymatrix{ 0 \ar[rr]^{a_0} && 1 \ar[rr]^{a_1} &&  2 \ar[rr]^{a_2} &&  3   \ar@/^1pc/[llllll]^{a_3} },
 \]
 
 \medskip\noindent and $\kI_{\bA}$ is the ideal of $\Mk \kQ_{\bA}$ generated by 
the set $R_{\bA}=\{a_0a_1a_2,$ $ a_1a_2a_3,$ $ a_3a_0\}$. Then $\Omega=\{2\}$ and $\bA^{\omega}=\Mk \kQ_{\bA^{\omega}}/\kI_{\bA^{\omega}}$ 
is the algebra such that

 \[
  \kQ_{\bA^{\omega}}: \xymatrix{ 0 \ar[rr]^{a_0} &&   1 \ar[rr]^{b_1} &&  3   \ar@/^1pc/[llll]^{a_{3}} },
 \]
 
 \medskip\noindent and $\kI_{\bA^{\omega}}$ is
 the ideal of $\Mk \kQ_{\bA^{\omega}}$ generated by 
the set $R_{\bA^{\omega}}=\{a_0b_1,  b_1a_3, a_3a_0\}$. Then  $(\kQ_{\bA^{\omega}} , Sp, R_{\bA^{\omega}})$ is  a skewed-gentle triple with $Sp=\{3\}$, and thus 
 $\bA^{\Omega}=\Mk \kQ_{\bA^{\Omega}}/\kI_{\bA^{\Omega}}$ is the algebra such that
  
  \[
 \kQ_{\bA^{\Omega}}:
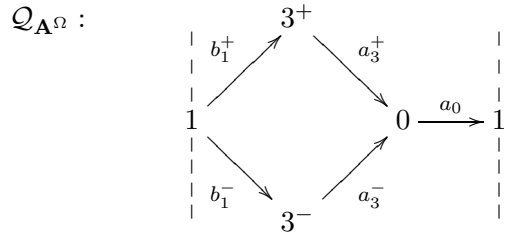
\xymatrix{
  &  &3^{+} \ar[rd]^{a_3^{+}}  &  &\\
 & 1 \ar@{--}[d] \ar@{--}[u]\ar[rd]_{b_{1}^{-}} \ar[ur]^{b_1^{+}} & & 0  \ar[r]^{a_0}
& 1\ar@{--}[d] \ar@{--}[u] \\
   & & 3^{-}\ar[ur]_{a_3^{-}}   & & }
\]

 \medskip\noindent where the dotted lines are identified and $\kI_{\bA^{\Omega}}$ 
is the ideal of $\Mk \kQ_{\bA^{\Omega}}$ generated by the set  
 $R_{\bA^{\Omega}}=\{b_1^{+}a_3^{+}-b_1^{-}a_3^{-}, a_3^{-}a_0,
  a_3^{+}a_0, a_0b_1^{-}, a_0b_1^{+}\}$. 
 
\end{exam}

\begin{exam}\label{exam3} Let $\bA=\Mk \kQ_{\bA}/\kI_{\bB}$ be the algebra such that 
\[
  \kQ_{\bA}=\mC_6: \xymatrix{ 0 \ar[r]^{a_0} & 1 \ar[r]^{a_1} &  2 \ar[r]^{a_2} &  3 \ar[r]^{a_3} &  4 \ar[r]^{a_4} &  5  \ar@/^1pc/[lllll]^{a_{5}} },
 \]
 
 \medskip\noindent and $\kI_{\bA}$ is the ideal of $\Mk \kQ_{\bA}$ 
generated by the set $R_{\bA}=\{a_0a_1a_2,$ $ a_1a_2a_3,$ $ a_3a_4a_5, a_4a_5a_0\}$. 
 Then $\Omega=\{2,5\}$ and $\bA^{\omega}=\Mk \kQ_{\bA^{\omega}}/\kI_{\bA^{\omega}}$ 
is the algebra such that

 \[
  \kQ_{\bA^{\omega}}: \xymatrix{ 0 \ar[r]^{a_0} &   1 \ar[r]^{b_1} &  3 \ar[r]^{a_3} &    4  \ar@/^1pc/[lll]^{b_{4}} },
 \]
 
 \medskip\noindent and $\kI_{\bA^{\omega}}$ is the ideal of $\Mk \kQ_{\bA^{\omega}}$ generated by the set $R_{\bA^{\omega}}=\{a_0b_1,$ $ b_1a_3,$ $a_3b_4,$ $b_4a_0\}$.
 
 Then  $(\kQ_{\bA^{\omega}}, Sp, R_{\bA^{\omega}})$ is  a skewed-gentle triple with $Sp=\{0,3\}$, and thus
 $\bA^{\Omega}=\Mk \kQ_{\bA^{\Omega}}/\kI_{\bA^{\Omega}}$ is the algebra such that
  
  \[
 \kQ_{\bA^{\Omega}}:\xymatrix{
  &  &3^{+} \ar[rd]^{a_3^{+}} & & 0^{+} \ar[rd]^{a_0^{+}} &\\
 & 1 \ar@{--}[d] \ar@{--}[u]\ar[rd]_{b_{1}^{-}} \ar[ur]^{b_1^{+}} & & 4  \ar[rd]_{b_4^{-}} \ar[ur]^{b_4^{+}}& 
& 1\ar@{--}[d] \ar@{--}[u] \\
   & & 3^{-}\ar[ur]_{a_3^{-}} &  & 0^{-} \ar[ur]_{a_0^{-}}& }
\]

 \medskip\noindent where the dotted lines are identified and $\kI_{\bA^{\Omega}}$ 
is the ideal of $\Mk \kQ_{\bA^{\Omega}}$ generated by the set  
 $R_{\bA^{\Omega}}=\{b_1^{+}a_3^{+}-b_1^{-}a_3^{-}, b_4^{+}a_0^{+}-b_4^{-}a_0^{-}, a_3^{\pm}b_4^{\pm},
  a_0^{\pm}b_1^{\pm}\}$. 
\end{exam}

\begin{prop}\label{sgentle}  Let $\bA$ be an algebra
which belongs to the class $\kD$ and which is not gentle.
Then $\bA$ is  derived equivalent to the skewed-gentle algebra $\bA^{\Omega}$.
\end{prop}

\begin{proof} 
 
 Define a complex $T_\bp=\oplus_{i\in \kQ_{0}} T_i$ of $\bA$-modules as follows. 
 Let $T_i:0\to A_i\to 0$ (in degree $0$) for  $i+1\in \kQ_0 \setminus \Omega$
 and $T_i:0\to A_i\to A_{i+1}\to 0$ (in degrees $1$ and $0$) for  $i+1\in \Omega$. 
 It is easy to check that the complex $T_\bp$ is tilting  and the
endomorphism algebra $\End_{\kD^{b}(\md{\bA})}(T_\bp)$ is isomorphic to  $\bA^{\Omega}$.

\end{proof}

\subsection{Proof of Theorem~\ref{t1}}\label{s23}

\begin{proof} If $\bA$ is a line algebra then in this situation Theorem~\ref{t1}  follows
from  \cite[Thm 1.1]{b} (see also \cite{g}). Thus we can assume that $\bA$
is a cycle algebra.
The implication $(\Rightarrow)$ follows from Lemma~\ref{derwild1}, Lemma~\ref{derwild2} and Lemma~\ref{derwild3}, whereas the implication $(\Leftarrow)$  follows from Proposition~\ref{sgentle}, Theorem~\ref{thm-skewed-gentle-tame}, Theorem~\ref{dereq} and Theorem~\ref{thm-gentle-tame}.

\end{proof}

\subsection{Proof of Theorem~\ref{t2}}\label{s24}

\begin{proof}
$(\Rightarrow)$. Since $\bA$ is a derived tame cycle algebra, it follows from Theorem~\ref{t1} that $\bA$ belongs to the class $\kD$. If $\kI$ is generated by relations of length two then $\bA$ is gentle, and for otherwise the statement follows from  Proposition~\ref{sgentle}.
$(\Leftarrow)$. If $\bA$ is gentle, then the statement follows from Theorem~\ref{thm-gentle-tame}. If $\bA$ is derived equivalent to some skewed-gentle algebra then the statement follows from Theorem~\ref{thm-skewed-gentle-tame} and Theorem~\ref{dereq}.
\end{proof}

\subsection{Proof of Corollary~\ref{c1}}\label{s25}

Following \cite{ka}, for a cycle algebra $\bA=\Mk \kQ/\kI$, we denote by $C(\bA)$ the set of equivalence classes
(with respect to cyclic permutation) of repetition-free cyclic paths $w_1\cdots w_n$ in $\kQ$ such
that $w_iw_{i+1}\in \kI$ for all $i$, where we set $n+1=1$. Moreover, we write $l(c)$ for the {\em length}
of a cycle $c\in C(\bA)$, i.e. $l(w_1\cdots w_n)=n$. Since $\bA$ is derived equivalent to the skewed-gentle algebra $\bA^{\Omega}$ by Proposition~\ref{sgentle}, 
it follows from \cite{cl} that $\kD_{sg}(\bA)\cong \kD_{sg}(\bA^{\Omega})\cong \kD_{sg}(\bA^{\omega})$. Hence we obtain by \cite{ka} that
$$\kD_{sg}(\bA)\cong \prod_{c\in C(\bA^{\omega})}\frac{\kD^{b}(\md{\Mk})}{[l(c)]}.$$
Since $| C(\bA^{\omega})|=1$ and $l(c)=| R_{\bA} |$ for $c\in C(\bA^{\omega})$.  This finishes the proof of Corollary~\ref{c1} .

\section{Acknowledgements}
This research was partly supported by CODI 
and Estrategia de Sostenibilidad 
2019-2020 (Universidad de Antioquia), and COLCIENCIAS-ECO\-PETROL (Contrato RC. No. 0266-2013)
and was accomplished during the visit of  the first and third authors at  the Instituto
 of Matem\'{a}ticas in the Universidad de Antioquia in Medell\'{\i}n, Colombia.
The hospitality offered by this university are gratefully acknowledged.

\end{document}